\documentclass[10pt]{article}
\usepackage{helvet}
\usepackage[all]{xy}
\usepackage{amssymb, amsthm}
\usepackage{epsfig}

\usepackage{hyperref}

\usepackage{amsmath}
\usepackage{mathrsfs}
\usepackage{algorithmic}
\usepackage{algorithm}

\usepackage{rotating}
\usepackage{multirow}
\usepackage{booktabs}

\newtheorem{theorem}{Theorem}
\newtheorem{lemma}{Lemma}

\newtheorem{definition}{Definition}

\newenvironment{packed_enum}{
\begin{enumerate}
\topsep=0pt plus 2pt minus 4pt
  \setlength{\itemsep}{1pt}
  \setlength{\parskip}{0pt}
  \setlength{\parsep}{0pt}
}{\end{enumerate}}

\newenvironment{packed_item}{
\begin{itemize}
\topsep=0pt plus 2pt minus 4pt
  \setlength{\itemsep}{1pt}
  \setlength{\parskip}{0pt}
  \setlength{\parsep}{0pt}
}{\end{itemize}}    
\begin{document}

\begin{center}
  \textbf{\LARGE Aggregation based on graph matching and inexact coarse grid solve for algebraic multigrid}\vspace*{3ex}\\
  Pawan Kumar\footnote{This work was funded by Fonds de la recherche scientifique (FNRS)(Ref: 2011/V 6/5/004-IB/CS-15) at Universit\'e Libre de Brussels, Belgique.}\\
  Service de M\'etrologie Nucl\'eaire \\
  Universit\'e libre de Bruxelles \\
  Bruxelles, Belgium \\
  \textsf{kumar.lri@gmail.com}\vspace*{0.5ex}\\
\end{center}
\section*{Abstract}
A graph based matching is used to construct aggregation based coarsening for
algebraic multigrid method. Effects of inexact coarse grid solve is analyzed
numerically for a highly discontinuous convection-diffusion
coefficient matrix, and for problems from Florida matrix market collection.
The proposed strategy is found to be more robust compared to a
classical AMG approach.
\section{Introduction}
We concern ourselves with the problem of solving large sparse linear
system of the form
\begin{eqnarray} \label{system} Ax = b,
\end{eqnarray}
arising from the cell centered finite volume discretization of the
convection diffusion equation as follows 
  \begin{eqnarray}\label{pde}
    \text{div}(\mathbf{a}(x)u)-\text{div}(\kappa(x)\nabla u)&=&f~\text{in}~\Omega, \notag \\
							    u&=&0~\text{on}~\partial\Omega_D, \\
				  \frac{\partial u}{\partial n}&=&0~\text{on}~\partial\Omega_N, \notag
  \end{eqnarray}
where $\Omega=[0, 1]^n$ ($n=2$, or 3),
$\partial\Omega_N=\partial\Omega\setminus\partial\Omega_D.$ The vector
field $\mathbf{a}$ and the tensor $\kappa$ are the given coefficients
of the partial differential operator. In 2D case, we have
$\partial\Omega_D=[0,1]\times \{0,1\}$, and in 3D case, we have
$\partial\Omega_D=[0,1]\times \{0,1\}\times [0,1]$. Other sources
include problems from Florida matrix market collection \cite{tim}; 
see Table~(\ref{Tab:FMM}) for a list of problems considered in this paper.

Currently, one of the most successful methods for these problems are
the multigrid methods (MG) \cite{rug, stu, tro}. The robustness of the multigrid method 
is significantly improved when they are used as a preconditioner in the Krylov subspace method \cite{saad96}. 
If $B$ denotes the MG preconditioner then the preconditioned linear system is the transformation 
of the linear system \eqref{system} to $B^{-1}Ax=B^{-1}b$. Here $B$ is a preconditioner which 
an approximation to the matrix $A$ such that the spectrum of $B^{-1}A$ is ``favourable'' for 
the faster convergence of Krylov methods. For instance, if the eigenvalues are clustered and are sufficiently close 
to one, then a fast convergence is observed in practice. Furthermore, the preconditioner B 
should be cheap to build and apply. With the advent of modern day multiprocessor and multicore 
era, the proposed method should have sufficient parallelism as well.

In multigrid like methods,
the problem is solved using a hierarchy of discretizations; the finest
grid is at the top of the hierarchy followed by coarser grids. The two complementary processes are: smoothening and coarse grid correction. The 
smoothers are usually chosen to be the classical relaxation methods such as
Jacobi, Gauss-Seidel, or incomplete LU methods \cite{saad96}. Analysis for model 
problems reveals that the smoothers efficiently eliminates the low frequency part of the error,
while the global correction which is obtained by solving a restricted
problem on the coarser grid damps the high frequency part of the error \cite{wie}. 
In fact, low frequency errors on the fine grid becomes high frequency error on the 
coarse grid leading to their efficient resolution on the coarser grid. It is therefore
 crucial to choose efficient smoothers and a coarse grid
solver. The classical geometric multigrid methods require informations
on the grid geometry and constructs a restriction operator and a
coarse grid. Since, a geometric multigrid method is closely related to the grid, the problem with 
nonlinearity can be resolved efficiently. 
But, for a complex grid, the applicability of the method becomes increasingly difficult.
On contrary, algebraic multigrid method defines the
necessary ingredients based solely on the coefficient matrix.
Much research have been devoted to algebraic multigrid methods and
several variants exists. 

In this paper, an aggregation based algebraic multigrid
is proposed. The aggregation is based on graph matching. This is
achieved by partitioning the graph of the matrix such that the
partitioned subgraphs are assumed to be the aggregates. Once a set of
aggregates is defined, the coarse grid is constructed from the Galerkin
formula. In \cite{kim}, authors use graph partitioner 
to form aggregates, and forward Gauss-Seidel with downwind numbering is used as 
pre- and post-smoother with the usual recursive multigrid method, where 
the coarsening is continued untill the number of unknowns in the coarse grid are less than 10. 
This approach may lead to a deep hierarchy of grids, thus making the method 
very recursive and less adapted to modern day multi-processor or 
multi-core environment. In \cite{ras}, similar graph based matching is used to form 
a coarse grid, and the classical recursive smoothed AMG approach is followed, however, 
here, ILUT \cite{saad96} is used for pre- and post-smoothing.

Our aim in this work is to propose a strategy that tries to avoid deep 
recursion but combines several different approaches as above. The strategy 
we adopt has the following ingredients:
\vspace{-5mm}
\begin{packed_item}
\item Coarsening based on graph matching
\item ILU(0) is the smoother with natural or nested dissection
  reordering 
\item Coarse grid equation is solved inexactly using ILUT
\end{packed_item}
\vspace{-5mm}
We show that the strategy proposed above is simple, easy to implement, and works well in practice for symmetric 
positive definite systems with large jumps in the coefficients. Solving 
a coarse grid inexactly leads to a faster and cheaper method. Indeed, 
a parallel incomplete coarse grid solve will be desirable, however,
in this work, we consider only the sequential version. We provide an estimate of heurestic 
coarse grid size and an estimate of a parameter involved in the inexact coarse grid solve. We compare 
our approach with a classical AMG \cite{not} with Gauss-Seidel smoothing and exact 
coarse grid solve. 

The rest of the paper is organized as follows. In section (2), we
discuss the classical coarsening strategy based on strength of connection, and the one based on graph matching. The numerical 
experiments are presented in section (3); the proposed method is compared with 
a classical AMG method on discontinous convection-diffusion problems and some problems 
from Florida matrix market collection \cite{tim}. Finally, section (4) concludes the paper. 

\section{Graph matching based aggregation for AMG}
In a typical two grid method, there are two complementary processes namely, a
smoother and a coarse grid correction step. When this strategy is repeated
by creating another coarser grid, then the method is known as
multigrid. During a fixed point iterative process, the high frequency 
components of the error or the so called
rough part are dealt with efficiently by a smoother. On the other hand, the low frequency
components of the error can only with dealt with globally by a method
that is ``connected'' globally (pertains to global fine grid). We
imagine the coefficients of the matrix as an approximation to some
function (Jacobian in case of nonlinear iteration). For an example,
assuming that a linear function is sufficiently smooth, an approximation
to the function with $N$ discrete points is close to an approximation
to the function with only $N/2$ (or even $N/4$) discrete points. Thus,
we solve the problem cheaply with $N/2$ grid points (i.e, on a coarser grid) and then we
interpolate the solution to obtain an approximation to the problem
defined on $N$ grid points. The error in the solution thus obtained
has rough components because they were not taken into
account properly while solving with the coarse solver, and this is
where smoother comes into play. This interplay of smoother and coarse
grid correction are complementary. For a more rigorous explanation, an
inclined reader is referred to \cite{rug, stu, tro} where tools from
Fourier analysis is used to explain why smoothening and coarse grid correction 
step works effectively for some problems.

In classical AMG, a set of coarse grid unknowns is selected and the
matrix entries are used to build interpolation rules that define the
prolongation matrix P, and the coarse grid matrix $A_c$ is computed
from the following Galerkin formula
\begin{eqnarray} \label{E:galerkin} A_c = P^{T}AP.
\end{eqnarray}
In contrast to the classical AMG approach, in aggregation based
multigrid, first a set of aggregates $G_{i}$ are defined. Let $N_c$
be the number of such aggregates, then the interpolation matrix $P$ is
defined as follows
\begin{equation*} \label{interp}
P_{ij} =
\begin{cases}
  1, &\text{if $i \in G_{j}$,}\\
  0, &\text{otherwise,}\\
\end{cases}
\end{equation*}
Here, $1 \le i \le N, \, 1 \le j \le N_c$, $N$ being the size of the
original coefficient matrix $A$. Further, we assume that the
aggregates $G_{i}$ are such that
\begin{equation} \label{aggr}
G_{i}\cap G_{j}=\phi,~ \text{for}~ i \neq j~ \text{and}~
\cup_iG_{i}= [1,N]
\end{equation}
Here $[1,\, N]$ denotes the set of integers from $1$ to $N$.
 Notice that the matrix $P$ defined above is an $N \times N_c$ matrix, but since it
has only one non-zero entry (which are ``one'') per row, the matrix
can be defined by a single array containing the indices of the non-zero entries. 
The coarse grid matrix $A_c$ may be computed as follows
\[
(A_c)_{ij} = \sum_{k \in G_i} \sum_{l \in G_j} a_{kl}
\]
where $1 \le i, \ j \le N_c$, and $a_{kl}$ is the $(k,l)th$ entry of $A$.

Numerous aggregation schemes have been proposed in the literature, but
in this paper we consider two of the aggregation schemes as follows
\vspace{-5mm}
\begin{description}
\item {\bf Aggregation based on strength of connection:} This approach
  is closely related to the classical AMG \cite{tro} where one first
  defines the set of nodes $S_i$ to which $i$ is strongly negatively
  coupled, using the Strong/Weak coupling threshold $\beta$:
  \[
  S_i = \{ \, j \neq i \mid a_{ij} < -\beta \ \text{max}|a_{ik}| \, \}.
  \]
  Then an unmarked node $i$ is chosen such that priority is given to
  the node with minimal $M_i$, here $M_i$ being the number of unmarked
  nodes that are strongly negatively coupled to $i$. For a complete algorithm 
 of the coarsening, the reader is referred to \cite{not}. 
\item {\bf Aggregation based on graph matching:} Several graph
  partitioning methods exists, notably, in software form
  \cite{kar3,saad2003,saad2002}. Aggregation for AMG is created by
  calling a graph partitioner with required number of aggregates as an
  input. The subgraph being partitioned are considered as
  aggregates. For instance, in this paper we use this approach by
  giving a call to the METIS graph partitioner routine
  METIS\_PartGraphKway with the graph of the matrix and number of
  partitions as input parameters. The partitioning information is obtained in
  the output argument ``part". The part array maps a given node to its
  partition, i.e., part($i$) = $j$ means that the node $i$ is mapped
  to the $jth$ partition. In fact, the part array essentially
  determines the interpolation operator $P$. For instance, we observe that 
  the ''part`` array is a discrete many to one map. Thus, the $i$th aggregate  
  $G_{i}=\text{part}^{-1}(i)$, where 
  \[ 
   \text{part}^{-1}(i) = \{ \, j \in [1,\, N] \enspace \mid \enspace 
\text{part}(j)=i\,\}
  \]
   Such graph matching techniques were explored in \cite{bra,kim,ras}. For
notational convenience, the method introduced in this paper will be called
  GMG (Graph matching based aggregation MultiGrid).
\end{description}
\vspace{-5mm}
Let $S$ denote the matrix which acts as a smoother in GMG method. The usual
choice of $S$ is a Gauss-Siedel preconditioner \cite{saad96}. However,
in this paper we choose ILU(0) as a smoother, we find that the choice
of ILU(0) as a smoother gives more robustness compared to Gauss-Siedel
method, however, at an additional storage cost. Another aspect that we
explore is to use only two grid
approach but with an incomplete coarse grid solve. That is, we use an
incomplete ILU($t$), where $t$ is the tolerance for dropping the
entries, see \cite{saad96}. The approximation $\tilde{A}_c$ of the
coarse grid operator $A_c$ is given as follows
\[
\tilde{A}_c = \tilde{L}_c\tilde{U}_c,~ \text{where},~
[\tilde{L}_c,\tilde{U}_c] = \text{ILUT}(A_c)
\]
where ILUT stands for ILU($t$). The reason for using
only two grid, and using an incomplete (and possibly parallel)
coarse grid solve is to avoid the recursion in the the typical AMG
method. It may be profitable to solve the coarse grid problem in
parallel and inexactly, when the problem size becomes large. This may
be achieved by a call to one of the several hybrid incomplete solvers
based on ILU \cite{ali} like approximation or by using a sparse
approximate inverse \cite{ben}.  The investigation with the parallel
inexact approximation of the coarse solver will be done in
future, and in this paper, we shall understand the qualitative
behavior such as the convergence and robustness of the proposed
strategy compared to a classical AMG approach found in \cite{not10}.

Let $M=PA_cP^{T}$ denote the coarse grid operator {\em interpolated}
to fine grid, then the two-grid preconditioner without post-smoothing is defined as follows
\begin{eqnarray} \label{twogrid}
  B = (S^{-1} + M^{-1} - M^{-1}AS^{-1})^{-1}. 
\end{eqnarray}
We notice that $M^{-1} \approx PA_{c}^{-1}P^{T}$, thus, an equation of the form 
$Mx=y$ is solved by first restricting $y$ to $y_c=P^T y$, then solving 
with the coarse matrix $A_c$ the following linear system: $A_cx_c = y_c$. 
Finally, prolongating the coarse grid solution $x_c$ to $x=Px_c$. Following 
diagram illustrates the two-grid hierarchy.
\begin{displaymath}
  \xymatrix{
         \cdots \bullet-\bullet-\bullet-\bullet\cdots \ar[d]_{\text{Restrict $y$ to $y_c:=P^{T}y$}} & \cdots \bullet-\bullet-\bullet-\bullet\cdots  \\
         \cdots \bullet-\bullet \cdots \ar[r]_{\text{Solve:$A_cx_c$ = $y_c$}} & \cdots \bullet-\bullet \cdots \ar[u]_{\text{Prolongate $x_c$ to $x:=P x_c$}}} 
\end{displaymath}
The preconditioner $B$ is similar to the combination preconditioner defined in
\cite{ach,kum}, where instead of defining a coarse grid operator a deflation
preconditioner is used. Thus, rather than satisfying a ``filtering
property'', the coarse grid operator satisfies the following ``approximate filtering condition (AFC)'' 
\[
AP \approx MP, \quad \text{(see Theorem (\ref{filter}) on page \pageref{filter})},
\]
where columns of interpolation matrix $P$ spans a subspace of dimension $N_c$. Here, we have considered the exact coarse grid solve, 
the inexact version is similar to the exact two-grid preconditioner \eqref{twogrid} defined 
above except that $M$ is replaced by $\tilde{M}=P\tilde{A}_cP^{T}$, and we denote the inexact 
two grid preconditioner by $\tilde{B}$. 
In Algorithm (\ref{Al:algo2G}), we present the complete iterative algorithm for the inexact case; the algorithm is 
essentially a slightly modified form of algorithm presented in Figure (2.6) in \cite{bar}. The two-grid methods can 
also be integrated in a similar way in an iterative accelerator other than GMRES, to integrate with other accelerators,  
see \cite{bar}.
\begin{algorithm} 
  \caption{\label{Al:algo2G} PSEUDOCODE TO SOLVE $Ax=b$, $A\in\mathbb{R}^{N \times N}$, $x,\, b\in\mathbb{R}^{N}$} 
\begin{algorithmic}
\STATE \begin{center} \underline{OBJECTIVE: To solve Ax = b} \end{center} 
\STATE \underline{SETUP PHASE} \\[0.3cm]
\STATE Call graph partitioner to get partitions in an array, say, part.
\STATE Use part array to form aggregates $G_{i}$ and the prolongation matrix $P$ (subgraphs are aggregates)
\STATE Create coarse grid matrix $A_c \in \mathbb{R}^{N_c \times N_c}$ as follows
\[
(A_c)_{ij} = \sum_{k \in G_i} \sum_{l \in G_j} a_{kl}.
\]
\STATE Factor the coarse grid matrix inexactly: $\tilde{A}_c=\text{ILUT}(A_c)$. Here ILUT is incomplete LU with tolerance.
\STATE Setup smoother: $S=L_0U_0= \text{ILU0}(A)$. Here ILU0 is incomplete LU with zero fill-in \\[0.3cm]
\STATE Define (not to be formed explicitely) two-grid preconditioner $\tilde{B}$ and $\tilde{M}$ as follows
\[ \tilde{B} = (S^{-1} + \tilde{M}^{-1} - \tilde{M}^{-1}AS^{-1})^{-1}, \quad \tilde{M}=P\tilde{A_c}P^{T} \] \\[0.3cm]
\STATE \underline{PRECONDITIONED GMRES ITERATION} \\[0.3cm]
\STATE $x_0$ is an initial guess
\FOR{$j = 1,2,\dots$}
  \STATE Solve $r$ from $\tilde{B}r=b-Ax_0$ (See SOLVE $\tilde{B}q=z$ function below)
  \STATE $v^{(1)}=r/\|r\|_{2}, \quad s:=\|r\|_2 e_1$
  \FOR{$i=1,2,\dots,m$}
    \STATE Solve $w$ from $\tilde{B}w=Av^{(i)}$ (See SOLVE $\tilde{B}q=z$ function below)
    \FOR{$k=1,\dots,i$}
      \STATE $h_{k,i}=(w,v^{(k)}), \quad w = w - h_{k,i}v^{(k)}$
    \ENDFOR
    \STATE $h_{i+1,i} = \|w\|_2, \quad v^{(i+1)} = w/h_{i+1,i}$
    \STATE apply $J_1,\dots,J_{i-1}$ on ($h_{1,i},\dots,h_{i+1,i}$)
    \STATE construct $J_i$, acting on the $i$th and $(i+1)$st component of $h_{.,i}$, \\
	   such that $(i+1)$st component of $J_ih_{.,i}$ is $0$
    \STATE set $s:=J_{i}s$
    \IF{$s(i+1)$ is small enough}
      \STATE UPDATE($\tilde{x},i$) and quit
    \ENDIF
    \ENDFOR
    \STATE UPDATE($\tilde{x},m$)
  \ENDFOR \\[0.3cm]
  \STATE \underline{UPDATE($\tilde{x},m$)}\\[0.2cm]
  \STATE Solve for $y$ in $Hy=\tilde{s}$. 
  Here upper $i \times i$ part of $H$ has $h_{i,j}$ 
    as its element. $\tilde{s}$ represents the first $i$ components of $s$
  \[ \tilde{x} = x^{(0)} + y_1v^{(1)} + y_2 v^{(2)} + \dots + y_iv^{(i)}, \quad s^{(i+1)} = \| b-A\tilde{x}\|_2 \]
\STATE If $\tilde{x}$ is accurate enough then quit else $x^{(0)}=\tilde{x}$ \\[0.3cm]
  \STATE \underline{SOLVE $\tilde{B}q=z$} \\[0.2cm]
  \STATE Solve $St=z$ (use $S=L_0U_0$), solve $\tilde{M}f=z$ (See SOLVE $\tilde{M}g=h$ function below), solve $\tilde{M}q=At$, 
  set $q=t+f-q$ \\[0.3cm]
  \STATE \underline{SOLVE $\tilde{M}g=h$} \\[0.2cm]
  \STATE set $h_c=P^{T}h$, Solve $\tilde{A}_c g_c = h_c$ (use $\tilde{A}_c=\tilde{L}\tilde{U}$), 
  $g=Pg_c$
\end{algorithmic}
\end{algorithm}

\subsection{Analysis of graph based two-grid method}
For any matrix $K$, let $K \succ 0$ denote that the matrix $K$ is
symmetric positive definite and we use the notation $K(:,j)$ to denote 
the $j$th column of $K$, whereas, $K(j,:)$ to denote the $j$th row of $K$. 
If $A \succ 0$, then the inner product $(\, , \,)_A$ defined by $(u,v)_{A} = u^{T}Av$ is a well defined inner
product, and it induces the energy norm $\| \, \|_A$ defined by $\|v\|_A
= (v,v)_A^{1/2}$ for any vector $v$. A matrix $K$ is called
A-selfadjoint if
\[
(Ku,v) = (u,Kv)_A,
\]
or equivalently if
\[
A^{-1}K^{T}A = K.
\]
For any matrix $K$, let span($K$) denote a set 
of all possible linear combination of the columns of the matrix $K$. Let $\|
x \|$ denote the Euclidean norm ($\sum_{i=1}^{n}x_i^{2})^{\frac{1}{2}}$.
In what follows, we assume that the matrix $P$ is orthonormalized, 
such that $P^TP=I$.
The basic linear fixed point method for solving the linear systems
\eqref{system} is given as follows
\begin{align*}
 x^{n+1} = x^{n} + B^{-1}(f-Ax^{n}) = (I-B^{-1}A)u^{n} + B^{-1}b
\end{align*}
Subtracting the equation above with the identity $x=x-B^{-1}Ax+B^{-1}b$ yields
the following equation for the error $e^{n+1}=u-u^{n}$
\begin{align*}
  e^{n+1} = (I-B^{-1}A)e^{n} = (I-B^{-1}A)^{2}e^{n-1} = \dots =
(I-B^{-1}A)^{n+1}e^{0}.
\end{align*}
Choosing $B$ as in Equation \eqref{twogrid}, we have the following relation 
\[ e^{n+1} = (I-S^{-1}A)^{n+1}(I-M^{-1}A)^{n+1}e^{0}\]
Thus, the quality of the preconditioner $B$ depends on how well the smoother
$S$ and the coarse grid preconditioner $M$ acts on the error.

In \cite{ach,kum,wagner1997a,wagner1997c,wagner1997b}, a composite preconditioner 
similar to the one in \eqref{twogrid} is proposed, where, the matrix $M$ is replaced 
by a preconditioner, say, $M_f$ that deflates the eigenvector corresponding to smaller eigenvalue. 
The preconditioner $M_f$ is constructed such that it satisfies a ``filtering property'' 
as follows
\[ M_ft = At,\]
where $t$ is a filter vector. 
In \cite{ach}, the filter vector is choosen to be a Ritz vector corresponding to smallest 
Ritz value in magnitude obtained after a couple of iterations of ILU(0) preconditioned 
matrix. In \cite{wagner1997a}, first the iteration is started 
with a fixed set of filter vector, and later the filter vector is changed adaptively using error 
vector. In \cite{kum}, authors fixed the filter vector to be a vector of all ones and show that 
the composite preconditioner is efficient for a range of convection-diffusion
type problems. In brief, for an effective method, the columns of the
interpolation matrix $P$ should approximate well the eigenvectors corresponding
to low eigenvalues. One possibility is to use $P$ to construct a deflation
preconditioner as shown in \cite{nab2004}.

The following theorem shows that preconditioner $M$ that corresponds to a coarse grid correction 
 step satisfies a more general \emph{but} approximate filtering condition. 
\begin{theorem} \label{filter}
If the coarse grid correction preconditioner $M=PA_cP^{T}$ and the
coarse grid operator $A_c=P^{T}AP$ are nonsingular, then following holds:
\vspace{-5mm}
\begin{packed_enum}
\item $MP \approx AP$
\item If $t=[1,1,\dots,1]$, then $t \in span(P)$ and dim(span(P)) = $N_c$
\end{packed_enum}
\vspace{-5mm}
\end{theorem}
\begin{proof}
We have 
\begin{align*}
(I-M^{-1}A)P &\approx P - P{A_c}^{-1}P^{T}AP \\
	     &= P - P{A_c}^{-1}A_c \\
             &= 0
\end{align*}
From equations \eqref{interp} and \eqref{aggr}, we find that $P$ is an 
$N \times N_c$ matrix, and the $j$th column of the matrix $P$ has a 
non-zero entry $P(i,j)=1$ if and only if $i \in G_j$. Since the aggregates $G_j's$
cover all the nodes in the set $[1,N]$, for all $i \in [1,N]$, there exists an 
aggregate $G_j$ such that $i = G_j(k)$ for some $k$, and consequently $P(i,j)=1$. Moreover, 
since the aggregates $G_j's$ do not intersect, such $j$ is unique. 
In other words, for each $i \in [1, N]$, there exists \emph{one and only one} column $P(:,j)$ of $P$ 
such that the $i$th entry of column $P(:,j)$ is 1. Hence we have 
\[ Pt= \sum_{1 \le i \le N_c}P(:,j) = t \]
and since each columns of $P$ are linearly independent we have dim(span(P)) = $N_c$.
\end{proof}
\vspace{-3mm}

\begin{theorem}
  If $A \succ 0$, then $A_c \succ 0$.
\end{theorem}\vspace{-5mm}
\begin{proof}
  We have $A_c = P^TAP$ and 
  \begin{eqnarray*}
     (A_cx, x) &=& (P^TAPx, x),~\text{for}~x \neq 0\\ 
    &=& (APx, Px), ~\text{for}~x \neq 0 \\
    &>& 0, ~\text{for}~x \neq 0. 
  \end{eqnarray*}
  Notice that we use the fact that $P$ is a boolean matrix, i.e., it
  has one and only one non-zero entry equal to ``one'' per row. Thus, $Px
  \neq 0$ for $x \neq 0, ~x \in R^{N_c}$. Hence the
  theorem.
\end{proof}
However, the global preconditioner corresponding to the coarse grid solve 
represented by $M=PA_cP^{T}$ or $M^{-1}A$ is not necessarily SPD. We have 
the following counter examples
\begin{theorem}
 $A\succ0$ does not imply that $M\succ0$ or $M^{-1}A \succ 0$. 
\end{theorem}
\begin{proof}
  Let $N=4$ be the size of $A$. Let there be two aggregates, $G_1 = \{ \, 1, 3 \, \}$ 
and $G_2 = \{ \, 2, 4 \, \}$, then the restriction operator $P^{T}$ is defined as follows
$P^{T}=\left[\begin{array}{cccc} 
  1 & 0 & 1 & 0 \\ 
  0 & 1 & 0 & 1
     \end{array}
   \right]$,
choosing $x^{T}=[1,0,-1,0]$, we have $P^{T}x=0$. Thus, we have 
\[ (Mx,x) \not> 0, \quad (M^{-1}Ax,x) \not> 0~\text{for all $x \neq 0$}. \qedhere \] 
\end{proof}
In literature, much results have been proved when the coefficient
matrix is a diagonally dominant $M-$matrix. We collect some relevant
results, and use them to understand the proposed method.
\begin{definition}
  Let $G(A)=(V,E)$ be the adjacency graph of a matrix $A \in
  \mathbb{R}^{N \times N}$.  The matrix $A$ is called irreducible if
  any vertex $i \in V$ is connected to any vertex $j \in
  V$. Otherwise, $A$ is called reducible.
\end{definition}
\begin{definition}
  A matrix $A \in \mathbb{R}^{N \times N}$ is called an $M-$matrix if
  it satisfies the following three properties: 
\vspace{-5mm}
\begin{packed_enum}
  \item $a_{ii}>0$ for $i=1,\dots,N$
  \item $a_{ij}\le 0$ for $i \neq j$, $i,j=1,\dots,N$
  \item $A$ is non-singular and $A^{-1} \ge 0$
  \end{packed_enum}
\end{definition}
\begin{definition}
  A square matrix $A$ is strictly diagonally dominant if the following
  holds
  \[
  |a_{ii}| > \sum_{j \neq i}|a_{ij}|, i=1,\dots,N
  \]
  and it is called irreducibly diagonally dominant if $A$ is
  irreducible and the following holds
  \[
  |a_{ii}| \geq \sum_{j \neq i}|a_{ij}|, i=1,\dots,N
  \]
  where strict inequality holds for atleast one $i$.
\end{definition}
A simpler criteria for $M-$matrix property is given by the following
theorem.
\begin{theorem}
  If the coefficient matrix $A$ is strictly or irreducibly diagonally
  dominant and satisfies the following conditions
 \vspace{-5mm} \begin{packed_enum}
  \item $a_{ii}>0$ for $i=1,\dots,N$
  \item $a_{ij}\le 0$ for $i \neq j$, $i,j=1,\dots,N$
  \end{packed_enum} \vspace{-5mm}
  then $A$ is an $M-$matrix.
\end{theorem}
\begin{theorem}[\cite{kim}] \label{Th:Mprop} If $A \in \mathbb{R}^{N
    \times N}$ is a strictly or irreducibly diagonally dominant
  $M-$matrix, then so is the coarse grid matrix $A_c=P^{T}AP$.
\end{theorem}
\vspace{-7mm}
\begin{proof}
  The theorem is proved in \cite{kim}.
\end{proof}
\begin{theorem}[\cite{mei}] \label{iluconv} If the coefficient matrix $A$ is symmetric
  $M-$matrix, and let $S=\tilde{L}\tilde{L}^{T}$ be the incomplete
  cholesky factorization, then the fixed point iteration with the
  error propogation matrix $I-S^{-1}A$ is convergent.
\end{theorem}
Theorem \ref{iluconv} above tells us that for an M-matrix, ILU(0) preconditioned method 
will be convergent by itself. However, the convergence is usually slow due to large 
iteration count with increasing problem size. Combining ILU(0) with a coarse grid 
correction leads to convergence rate which depends mildly on the problem size.
Following result shows that the inexact factorization is as stable as
the exact factorization of the coarse grid operator.
\begin{theorem}\label{Th:stable}
  If the given coefficient matrix $A$ is a symmetric irreducibly
  diagonally dominant $M-$matrix, and if the inexact coarse grid
  operator $\tilde{A}_c$ is based on incomplete LU factorization as follows 
  \[
  \tilde{A}_c = \textup{CHOLINC}(A_c),
  \]
where \textup{CHOLINC} is the incomplete Cholesky factorization, then the
construction of $\tilde{A}_c$ is atleast as stable as the construction 
of an exact decomposition of $A_c$ without pivoting.
\end{theorem} \vspace{-7mm}
\begin{proof}
  If the original matrix $A$ is symmetric and irreducibly diagonally
  dominant $M-$matrix, then Theorem \ref{Th:Mprop} tells us that the
  coarse grid operator $A_c$ obeys the same property. Now, $A_c$ being
  an $M-$matrix, Theorem 3.2 in \cite{mei} tells us that $\tilde{A}_c$
  defined above is as stable as the exact Cholesky factorization of $A_c$.
\end{proof}
For a diagonally dominant $M$-matrix, pivoting is rarely needed. However, 
pivoting generally improves the stability of incomplete LU type
factoriations. This is the reason why we use incomplete LU with
pivoting, namely, ILUT function of MATLAB. Moreover, using ILUT will lead 
to a method suitable for unsymmetric matrices that are not necessarily diagonally dominant.
We refer the curious reader to \cite{hig} for a small $2 \times 2$ example where
pivoting would be essential to obtain stable triangular factorization.

For problems with jumping coefficients, the ratio of maximum and minimum entry 
of the coefficient matrix can provide some useful bounds as shown in the
theorem below.
\begin{lemma}[page 7, \cite{kim}]\label{Th:boundCond}
 Let A be a symmetric $N \times N$ matrix with eigenvalues $\lambda_1(A) \leq
\dots \leq \lambda_N(A)$ arranged in nondecreasing order, then the following
holds
\[
\lambda_1(A) \leq min_i\{a_{ii}\} \leq max_i \{a_{ii}\} \leq \lambda_N(A). 
\]
In particular, if $A\succ0$, then cond(A) is bounded below by
$\frac{max_i\{a_{ii}\}}{min_i\{a_{ii}\}}$.
\end{lemma}\vspace{-7mm}
\begin{proof}
The proof follows by writing the following expression
\[
\lambda_1(A) = max_{\| x \| = 1}\{ x^{T}Ax \}, \quad  \lambda_n(A) = max_{\| x
\| = 1}\{ x^{T}Ax \},
\] 
and by setting $x$ as the $i$th column of the identity matrix $I$.
\end{proof}
In Table~\ref{T:cond_est}, we check our estimates on the problems considered
in Section~\ref{S:numExp}. We compare the estimate obtained above to the
that given by the "condest`` function of MATLAB. We find that the estimates
obtained using Theorem~\ref{Th:boundCond} are not close, nevertheless, they do
indicate the increase in the order of magnitude of the condition number with
increasing jumps. When a given coefficient matrix
is SPD, we may indeed use this as a heurestic in determining the quality of the
inexact factorization, see last column of Table~\ref{T:cond_est}. Let $l_{ij}$ denote the 
$(i,j)th$ entry of the lower triangular matrix $L$. For a 
given SPD matrix, it is easily verified that the condition number estimate for the 
incomplete factorization $\tilde{L}\tilde{L}^T$ is given by $\frac{\|L(N,:)\|_2^{2}}{l_{11}^2}$. 
For the unsymmetric matrix, the estimate is given by 
$\frac{max_{i,j}\{ \sum_{i,j \leq i} \tilde{l}_{ij}\tilde{l}_{ji}\}}{ min_{i,j}\{ \sum_{i,j \leq i} \tilde{l}_{ij}\tilde{l}_{ji}\}}$.
These estimates are cheap to compute and can be used as a fault-tolerant mechanism while 
using inexact factorization, for instance, during inexact coarse grid solve as 
implemented in this paper.

\begin{table}
  \caption{ Comparison matlab condest function with the estimate in Lemma~
\ref{Th:boundCond} for exact and inexact factorization for JUMP3D problems
defined in Section~\ref{S:numExp}. Here $\tilde{a}_{ii}$ is the $(i,i)$th
diagonal entry of the inexact coarse grid operator $\tilde{A}_c$ factorization with drop tolerance of
$10^{-4}$}. 
  \label{T:cond_est}
  \begin{center}
    \begin{tabular}{lllllll}
      \toprule
\addlinespace
    $h$  &
\multicolumn{2}{c}{condest}
&\multicolumn{2}{c}{$\frac{max_i\{a_{ii}\}}{min_i\{a_{ii }\}}$} &  
\multicolumn{2}{c}{$\frac{max_i\{\tilde{a}_{ii}\}}{min_i\{\tilde{a}_{ii }\}}$} 
\\
\addlinespace
      \cmidrule(lr){2-3} \cmidrule(lr){4-5} \cmidrule(lr){6-7} 
  & $\kappa = 10^{3}$ & $\kappa = 10^{5}$ & $\kappa = 10^{3}$ & $\kappa =
10^{5}$ & $\kappa = 10^{3}$ & $\kappa = 10^{5}$\\
\addlinespace
      \midrule
\addlinespace
      $1/20$       & $1.1 \times 10^{6}$   & $7.8 \times 10^6$ &  $7.5 \times
10^3$ & $7.5 \times 10^5$   & $9.3\times 10^3$ &   $9.3\times 10^5$       \\
      $1/30$       &  $3.7 \times 10^6$  & $7.7 \times 10^9$  &  $9.0 \times
10^3$ & $9.0 \times 10^5$ & $1.0 \times 10^4$ & $1.0 \times 10^6$ \\
      $1/40$       &  $6.4 \times 10^6$  & $4.7 \times 10^9$  &  $9.0 \times
10^3$ & $9.0 \times 10^5$  & $1.0 \times 10^4$ &  $1.0 \times 10^6$      \\
\addlinespace
      \bottomrule
    \end{tabular}
  \end{center}
\end{table}
In \cite{not2007}, convergence analysis of perturbed two-grid and multigrid method
was done. In the context of domain decomposition methods, in \cite{bra1998}, 
numerical and theoretical analysis suggests the advantages of using inexact solves. 
However, a systematic study of the scalability of inexact coarse grid solve has 
been missing.

\section{Numerical experiments} \label{S:numExp}
All the numerical experiments were performed in MATLAB with double
precision accuracy on Intel core i7 (720QM) with 6 GB RAM. For
comparison, we use the aggregation based AMG (AGMG) software available
at \cite{not}. The AGMG software is a Fortran mex file, on the other hand, the
AMG method introduced in this paper, namely, GMG, is written
completely in MATLAB.  For GMG, the iterative accelerator used is
GMRES available at \cite{saad_soft}, the code was changed such that the 
stopping is based on the decrease of the 2-norm of the relative residual. For
AGMG, GCR method is used \cite{not10}. For both
GMRES and GCR, the maximum number of iterations allowed is 600, and no
restart is done. The stopping criteria is the decrease of the relative
residual below $10^{-7}$, i.e., when
\[
\frac{\|b-Ax_k\|}{\|b\|} < 10^{-7}.
\]
Here $b$ is the right hand side and $x_k$ is an approximation to the
solution at the $k$th step.

\begin{table}
  \caption{Notations used in tables of numerical experiments }
  \label{not}
  \begin{center}
    \begin{tabular}{ll}
      \toprule
\addlinespace
      Notations & Meaning                                                           \\
\addlinespace
      \midrule
\addlinespace
      $h$       & Discretization step                                   \\
      $N$        & Size of the original matrix                                \\
      $N_c$        & Size of the coarse grid matrix                                \\ 
      its       & Iteration count                                                   \\
      time      & Total CPU time (setup plus solve) in seconds       \\
      $cf$        & coarsening factor \\
      ME        & Memory allocation problem                                     
              \\
      F1       & AGMG returned flag 1, see \cite{not} \\
      SPD        & Symmetric positive definite \\
      NA        & Not applicable                                                    \\
      NC        & Did not converged                                                     \\
      -         & Data not available                      \\
      GMG-NO  & Graph based matching for AMG, smoother has ND ordering         \\
      GMG-ND   &  Graph based matching for AMG, smoother has natural ordering  \\
      EGMG-ND       &   Graph based matching for AMG, exact coarse grid, smoother has ND ordering                                          \\
      AGMG         & Classical AMG, see \cite{not} \\
     $cf$     & Coarsening factor, $cf = n_c/n$ for $n$ and $n_c$ no. of discrete points \\ 
              & for uniform fine and coarse grid respectively. \\
\addlinespace
      \bottomrule
    \end{tabular}
  \end{center}
\end{table}

\subsection{Test cases}
\begin{description}
\item {\bf Convection-Diffusion:} Our primary test case is the
convection-diffusion Equation \eqref{pde} defined on page
\pageref{pde}. We use the notation DC to indicate that the problems are discontinous. We consider a test case as follows
\begin{description}
\item {\bf DC1, 2D case:} The tensor $\kappa$ is isotropic and
  discontinuous. The domain contains many zones of high permeability
  that are isolated from each other. Let $[x]$ denote the integer
  value of $x$. For two-dimensional case, we define $\kappa(x)$ as
  follows:
  \begin{eqnarray*}
    \kappa(x)=\left\{
      \begin{array}{ll}
        10^3\ast ([10\ast x_2]+1),~ \text{if} \hspace{0.2cm}[10\ast x_i]\equiv0
        \hspace{0.12cm}~(mod~2),~i=1, 2,\\
        1,~ \text{otherwise}.
      \end{array}
    \right.
  \end{eqnarray*}
  The velocity field $\mathbf{a}$ is kept zero.
  We consider a $n \times n$ uniform grid where $n$ is the number of discrete points along each spatial directions.
  \item {\bf DC1, 3D case:} For three-dimensional case, $\kappa(x)$ is defined as follows:
  \begin{eqnarray*}
    \kappa(x)=\left\{
      \begin{array}{ll}
        10^3\ast ([10\ast x_2]+1),\hspace{0.3cm} \text{if} \hspace{0.2cm}[10\ast x_i]\equiv 0
        \hspace{0.12cm} (mod~2) \hspace{0.12cm}, \hspace{0.12cm} i=1, 2, 3,\\
        1,~\text{otherwise}.
    \end{array}
  \right.
\end{eqnarray*}
  Here again, the velocity field $\mathbf{a}$ is kept zero. We consider a $n
\times n \times n$ uniform grid. The jump in the diagonal entries of the coefficient matrix 
is shown in Figure (\ref{F:nonZeros}).
\item {\bf DCC1, 2D case:} Same as DC1, 2D case above, except that the velocity is non-zero and it is given as
$a(x)=(1000,1000)$.
\item {\bf DCC2, 3D case:} Same as DC1, 3D case above, except that the velocity 
$a(x)=(1000,1000,1000)$.
\end{description}
\item {\bf Florida matrix market collection:} The list of Florida
  matrix matrices are shown in Table (\ref{Tab:FMM}). As we observe,
  all the problems are symmetric positive definite steaming from wide
  range of applications. For more on the properties of these matries,
  the reader is referred to \cite{tim}.
  \begin{table}
    \caption{\label{Tab:FMM} Forida matrix market matrices. Here SPD stands for 
      symmetric positive definite.}
    \begin{center}
      \begin{tabular}{lllll}
        \toprule
\addlinespace
        \textbf{Matrices} & \textbf{Kind}                  & \textbf{SPD} & \textbf{size} & \textbf{non-zeros} \\
\addlinespace        
\midrule
\addlinespace
        gyro\_m           & Model reduction problem        & Yes          & 17361         & 340K               \\
        bodyy4            & Structural problem             & Yes          & 17546         & 121K               \\
        nd6k              & 2D/3D problem                  & Yes          & 18000         & 6.8M              \\
        bodyy5            & Structural problem             & Yes          & 18589         & 128K               \\
        wathen100         & Random 2D/3D problem           & Yes          & 30401         & 471K               \\
        wathen120         & Random 2D/3D problem           & Yes          & 36441         & 565K               \\
        torsion1          & Duplicate optimization problem & Yes          & 40000         & 197K               \\
        obstclae          & Optimization problem           & Yes          & 40000         & 197K               \\
        jnlbrng1          & Optimization problem           & Yes          & 40000         & 199K              \\
        minsurfo          & Optimization problem           & Yes          & 40806         & 203K              \\
        gridgena          & Optimization problem           & Yes          & 48962         & 512K               \\
        crankseg\_1       & Structural problem             & Yes          & 52804         & 10M                \\
        qa8fk             & Acoustic problem               & Yes          & 66127         & 1M                 \\
        cfd1              & Computational fluid dynamics   & Yes          & 70656         & 1.8M               \\
        finan512          & Economic problem               & Yes          & 74752         & 596K               \\
        shallow\_water1   & Computational fluid dynamics   & Yes          & 81920         & 327K               \\
        2cubes\_sphere    & Electromagnetic problem        & Yes          & 101492        & 1.6M               \\
        Thermal\_TC       & Thermal problem                & Yes          & 102158        & 711K               \\
        Thermal\_TK       & Thermal problem                & Yes          & 102158        & 711K               \\
        G2\_circuit       & circuit simulation             & Yes          & 150102        & 726K               \\
        bcsstk18          & structural problem             & Yes          & 11948        & 149K               \\
        cbuckle           & structural problem             & Yes          & 13681        & 676K               \\
        Pres\_Poisson     & Computational fluid dynamics   & Yes          & 14822        & 715K   \\         
\addlinespace
        \bottomrule
      \end{tabular}
    \end{center}
  \end{table}
\end{description}

\subsection{Comments on numerical results}
Two version of GMG are
shown, namely, GMG-NO which stands for GMG where smoother has
natural ordering, and GMG-ND stands for GMG with smoother having
nested dissection ordering. In particular, for GMG-ND, we first
apply the nested dissection reordering and then the smoother is
defined. We observe that after applying nested dissection reordering,
the smoother which is ILU(0) in our case can be computed and applied
in parallel.  Since, in ILU(0), no pivoting is done, parallelizing
ILU(0) after ND ordering leads to a parallel smoother. Certainly, not
much parallelism is expected when the smoother is applied with natural
ordering of unknowns. As mentioned before, for the coarse grid solve, we use
ILU($10^{-4}$) to solve it inexactly. We do this inexact solve to see
the effect of inexact solve in the iteration count and time. For AGMG, Gauss-Seidel smoothing is
used, and the choice of the coarse grid is based on the strength of
connection between nodes. Moreover, in AGMG, usual multilevel recursive approach is
followed, i.e., going down the grid heirarchy untill the coarse grid is small enough (or it stops 
when it does not satisfy certain criteria) to be 
solved exactly.

We recall here that our aim is to compare the classical
multi-grid approach implemented in AGMG with the two grid approach of
GMG with the following ingredients
\vspace{-5mm}
\begin{packed_item}
\item Coarse grid based on graph matching (call to METIS)
\item ILU(0) is the smoother (built in MATLAB)
\item Coarse grid equation is solved inexactly (using built in
  ILU($t$) routine in MATLAB)
\end{packed_item}
\vspace{-5mm}

In Tables (\ref{T:large_cf_2.5}), (\ref{T:large_cf_3}), (\ref{T:large_cf_3.5}),
(\ref{T:large_cf_4}), (\ref{T:large_cf_4.5}), (\ref{T:large_cf_5}),
(\ref{T:large_cf_5.5}),
 (\ref{T:large_cf_6}), (\ref{T:large_cf_7}), and (\ref{T:large_cf_8}), we have 
shown the iteration count and the total CPU time: setup time plus solve time,  
for values of $cf$ ranging from 3 to 8. For the values of $cf$ lying between 
3 and 8, we can locate the value of $cf$ for which the CPU time is observed to 
be lowest, to do so, we had to vary $cf$ so that we do not miss a value of 
$cf$ for which the CPU time could be lowest. For 2D problems, we find that 
the AGMG method is several times faster than the GMG based methods. However, for 
3D problems, the AGMG method does not converge at all. In contrast, GMG methods 
converges and shows mesh independent convergence rates for all values of $cf$. 
Considering 2D case first: the least CPU time for GMG-NO method is observed for $cf=3$. 
For GMG-ND method, the least CPU time is observed for $cf=2.5$ except for 2D $1200 \times 1200$ problem for which the least 
CPU time occurs for $cf=3$. On the other hand, for EGMG-NO method, the least CPU time was observed 
when we had $cf=4$ for $800 \times 800$ and $1200 \times 1200$ problem, and the least CPU time 
for $1200 \times 1200$ problem was obtained when $cf$ was equal to 6.
For a smaller size 3D problem of $70 \times 70
\times 70$, the iteration count rather increases rapidly with the increasing value of $cf$. For 
smaller problems, it seems that a finer coarse grid is necessary. The reason
for high iteration count may be due to the fact that for a given partial
differential equation, the coefficient matrix becomes smoother 
as the resolution of the mesh is increased.
In Figures \ref{F:curvePres} and \ref{F:curveJUMP}, 
we find that the convergence curve for the respective exact and inexact 
methods are very similar, this also suggests a similar spectrum and 
probably a similar condition number. To find how close the 
approximated coarse matrix is to the exact coarse operator, in Table \ref{T:norm}, 
we compare the relative error $\|LU-\tilde{L}\tilde{U}\|/\|LU\|$ for both natural and ND ordering. 
We find that the relative error is quite small, this is the reason why 
the direct and indirect versions converges in similar number of iterations. 
But since the inexact methods are relatively fast to build and apply,
we save significant number of CPU time and storage requirement, see Figure 
(\ref{F:nonZeros}) where an inexact solve needs about 10 times less storage compared to 
the exact solver.

In Tables (\ref{T:large_con_cf_2.5}), (\ref{T:large_con_cf_3}),
(\ref{T:large_con_cf_3.5}), (\ref{T:large_con_cf_4}),
(\ref{T:large_con_cf_4.5}),
(\ref{T:large_con_cf_5}), (\ref{T:large_con_cf_5.5}),
 (\ref{T:large_con_cf_6}), and (\ref{T:large_con_cf_7}), we have 
similar plots for the test case DCC1. For this problem, we find that the AGMG
method converges much faster 
compared to the GMG methods for both 2D and 3D problems, exception being the
$120 \times 120 \times 120$ problem where the method fails to converge. However,
we remind ourselves that GMG methods are implemented in Matlab, and thus 
they are expected to be faster when they are implemented in lower level
languages such as Fortran or C. Thus, our prediction is that even for these
problems where GMG shows larger CPU time, an implementation in Fortran may
have the convergence time comparable with that of AGMG. Notably, for this test
case, the iteration count decreases even more rapidly (compared to test case
DC1) with the increase in the size of the problem.

The rule of thumb in the choice of coarse grid size is to increase the $cf$
value proportionally with the increasing size of the problem. For a smaller 
size problem with discontinous coefficients such as DC1 and DCC1, it is good to 
keep the $cf$ value small. The choice of drop tolerance in ILUT to be $10^{-6}$ 
worked well in practice and we did not encounter any breakdown. However, in
case when the jumps are large, the fault tolerant mechanism such as the one
discussed in the analysis section can be used.

Finally, in Table (\ref{table1}), we show some experiments with the Florida
matrix market problems. We fixed the coarse grid size to be 4096. In general,
for most of the problems, we find that the two-grid method
is faster compared to AGMG, exception being, torsion1, obstclae,
jnlbrng1, minsurfo, qa8fk, and shallow\_water, where AGMG is about five times
faster. For rest of the problems, GMG methods shows more robustness compared
to AGMG. Comparing GMG-NO to GMG-ND, we find that GMG-NO converges faster
with few exceptions. In Table (\ref{T:table_florida2}), we present the numerical
results with $cf=2.5$. A detailed investigation of the best coarse grid size
for these problems deserves more effort and detailed study. 
%
\begin{table}
  \caption{Numerical results for DC1 problem with $cf=2.5$ using GMRES(30).}
  \label{T:large_cf_2.5}
  \begin{center}
    \begin{tabular}{llllllllll}
     \toprule
\addlinespace
      matrix & $h$ & \multicolumn{2}{c}{GMG-ND} & \multicolumn{2}{c}{GMG-NO}
& \multicolumn{2}{c}{EGMG-NO} & \multicolumn{2}{c}{AGMG} \\
\addlinespace
      \cmidrule(lr){3-4} \cmidrule(lr){5-6} \cmidrule(lr){7-8}
\cmidrule(lr){9-10}
\addlinespace
             &       & its                          & time                      
  & its & time & its & time & its & time  \\
\addlinespace
\midrule
\addlinespace

                & 1/800 & 23 & 25.5  & 17 & 17.4 & 16 & 38.7 & 27 & 6.3   \\
         2D & 1/1000 & 23 & 45.5  & 18 & 30.3 & 18 & 82.0 & 37 & 14.1  \\
                & 1/1200 & 24 & 70.8 & 20 & 48.2 & 18 & 159.5 & 37 & 18.2 \\
&  & &  &  &  \\
\addlinespace

                & 1/70 & 145 & 65.8  & 21 & 21.0  & 14 & 163.0 &  NC & NA \\
         3D & 1/100 & 191 & 306.5 & 26 & 141.0  & ME & NA &  NC & NA
\\
                & 1/120 & 147 & 598.0   & 45 & 379.2 & ME & NA &  NC & NA
\\
\addlinespace
      \bottomrule
    \end{tabular}
  \end{center}
\end{table}
%
%
%
\begin{table}
  \caption{Numerical results for DC1 problem with $cf=3$ using GMRES(30).}
  \label{T:large_cf_3}
  \begin{center}
    \begin{tabular}{llllllllll}
     \toprule
\addlinespace
      matrix & $h$ & \multicolumn{2}{c}{GMG-ND} & \multicolumn{2}{c}{GMG-NO}
& \multicolumn{2}{c}{EGMG-NO} & \multicolumn{2}{c}{AGMG} \\
\addlinespace
      \cmidrule(lr){3-4} \cmidrule(lr){5-6} \cmidrule(lr){7-8}
\cmidrule(lr){9-10}
\addlinespace
             &       & its                          & time                      
  & its & time & its & time & its & time  \\
\addlinespace
\midrule
\addlinespace

                & 1/800 & 28 & 24.7  & 20 & 16.8 & 19 & 22.8 & 27 & 6.3   \\
         2D & 1/1000 & 29 & 50.4  & 19 & 25.6 & 19 & 25.5 & 37 & 14.1  \\
                & 1/1200 & 26 & 65.8 & 19 & 37.6 & 20 & 70.8 & 37 & 18.2 \\
&  & &  &  &  \\
\addlinespace

                & 1/70 & 165 & 66.0  & 20 & 14.3  & 15 & 47.5 &  NC & NA \\
         3D & 1/100 & 191 & 314.1 & 25 & 94.2  & 16 & 839.2 &  NC & NA
\\
                & 1/120 & 165 & 417.0   & 35 & 236.8 & ME & ME &  NC & NA
\\
\addlinespace
      \bottomrule
    \end{tabular}
  \end{center}
\end{table}
%
%
\begin{table}
  \caption{Numerical results for DC1 problem with $cf=3.5$ using GMRES(30).}
  \label{T:large_cf_3.5}
  \begin{center}
    \begin{tabular}{llllllllll}
     \toprule
\addlinespace
      matrix & $h$ & \multicolumn{2}{c}{GMG-ND} & \multicolumn{2}{c}{GMG-NO}
& \multicolumn{2}{c}{EGMG-NO} & \multicolumn{2}{c}{AGMG} \\
\addlinespace
      \cmidrule(lr){3-4} \cmidrule(lr){5-6} \cmidrule(lr){7-8}
\cmidrule(lr){9-10}
\addlinespace
             &       & its                          & time                      
  & its & time & its & time & its & time  \\
\addlinespace
\midrule
\addlinespace

                & 1/800 & 33 & 30.6  & 28 & 22.7 & 22 & 19.9 & 27 & 6.3   \\
         2D     & 1/1000 & 36 & 50.9  & 23 & 28.6 & 22 & 33.8 & 37 & 14.1  \\
                & 1/1200 & 34 & 72.7 & 23 & 42.2 & 22 & 54.8 & 37 & 18.2 \\
&  & &  &  &  \\
\addlinespace

                & 1/70 & 153 & 55.7  & 23 & 11.6  & 17 & 21.5 &  NC & NA \\
         3D     & 1/100 & 200 & 221.8 & 32 & 58.6  & 23 & 254.8 &  NC & NA
\\
                & 1/120 & 153 & 339.5   & 33 & 119.5 & 19 & 740.2 &  NC & NA
\\
\addlinespace
      \bottomrule
    \end{tabular}
  \end{center}
\end{table}
%
%
\begin{table}
  \caption{Numerical results for DC1 problem with $cf=4$ using GMRES(30).}
  \label{T:large_cf_4}
  \begin{center}
    \begin{tabular}{llllllllll}
     \toprule
\addlinespace
      matrix & $h$ & \multicolumn{2}{c}{GMG-ND} & \multicolumn{2}{c}{GMG-NO}
& \multicolumn{2}{c}{EGMG-NO} & \multicolumn{2}{c}{AGMG} \\
\addlinespace
      \cmidrule(lr){3-4} \cmidrule(lr){5-6} \cmidrule(lr){7-8}
\cmidrule(lr){9-10}
\addlinespace
             &       & its                          & time                      
  & its & time & its & time & its & time  \\
\addlinespace
\midrule
\addlinespace

                & 1/800 & 44 & 34.8  & 26 & 20.6 & 24 & 19.7 & 27 & 6.3   \\
         2D & 1/1000 & 37 & 49.8  & 25 & 30.1 & 25 & 33.2 & 37 & 14.1  \\
                & 1/1200 & 37 & 72.3 & 28 & 51.1 & 27 & 53.8 & 37 & 18.2 \\
&  & &  &  &  \\
\addlinespace

                & 1/70 & 164 & 57.6  & 84 & 29.5  & 55 & 26.3 &  NC & NA \\
         3D & 1/100 & 212 & 235.8 & 27 & 44.5  & 19 & 116.3 &  NC & NA
\\
                & 1/120 & 168 & 346.2   & 31 & 94.8 & 21 & 326.3 &  NC & NA
\\
\addlinespace
      \bottomrule
    \end{tabular}
  \end{center}
\end{table}
%
%
\begin{table}
  \caption{Numerical results for DC1 problem with $cf=4.5$ using GMRES(30).}
  \label{T:large_cf_4.5}
  \begin{center}
    \begin{tabular}{llllllllll}
     \toprule
\addlinespace
      matrix & $h$ & \multicolumn{2}{c}{GMG-ND} & \multicolumn{2}{c}{GMG-NO}
& \multicolumn{2}{c}{EGMG-NO} & \multicolumn{2}{c}{AGMG} \\
\addlinespace
      \cmidrule(lr){3-4} \cmidrule(lr){5-6} \cmidrule(lr){7-8}
\cmidrule(lr){9-10}
\addlinespace
             &       & its                          & time                      
  & its & time & its & time & its & time  \\
\addlinespace
\midrule
\addlinespace

                & 1/800 & 43 & 32.7  & 27 & 20.5 & 27 & 21.1 & 27 & 6.3   \\
         2D & 1/1000 & 44 & 53.9  & 28 & 33.9 & 27 & 33.4 & 37 & 14.1  \\
                & 1/1200 & 41 & 74.0 & 27 & 47.5 & 29 & 55.6 & 37 & 18.2 \\
&  & &  &  &  \\
\addlinespace

                & 1/70 & 170 & 58.7  & 89 & 31.0  & 70 & 26.6 &  NC & NA \\
         3D & 1/100 & 205 & 215.4 & 28 & 40.9  & 22 & 63.3 &  NC & NA
\\
                & 1/120 & 157 & 302.3   & 34 & 83.5 & 24 & 165.0 &  NC & NA
\\
\addlinespace
      \bottomrule
    \end{tabular}
  \end{center}
\end{table}
%
%
\begin{table}
  \caption{Numerical results for DC1 problem with $cf=5$ using GMRES(30).}
  \label{T:large_cf_5}
  \begin{center}
    \begin{tabular}{llllllllll}
     \toprule
\addlinespace
      matrix & $h$ & \multicolumn{2}{c}{GMG-ND} & \multicolumn{2}{c}{GMG-NO}
& \multicolumn{2}{c}{EGMG-NO} & \multicolumn{2}{c}{AGMG} \\
\addlinespace
      \cmidrule(lr){3-4} \cmidrule(lr){5-6} \cmidrule(lr){7-8}
\cmidrule(lr){9-10}
\addlinespace
             &       & its                          & time                      
  & its & time & its & time & its & time  \\
\addlinespace
\midrule
\addlinespace

                & 1/800 & 46 & 34.4  & 32 & 24.4 & 30 & 24.2 & 27 & 6.3   \\
         2D & 1/1000 & 47 & 56.8  & 32 & 37.9 & 34 & 39.2 & 37 & 14.1  \\
                & 1/1200 & 51 & 86.8 & 30 & 53.8 & 29 & 53.6 & 37 & 18.2 \\
&  & &  &  &  \\
\addlinespace

                & 1/70 & 284 & 94.2  & 199 & 63.3  & 193 & 62.9 &  NC & NA \\
         3D & 1/100 & 206 & 280.9 & 30 & 55.2  & 23 & 60.5 &  NC & NA
\\
                & 1/120 & 168 & 335.3   & 40 & 92.2 & 26 & 122.7 & NC & NA
\\
\addlinespace
      \bottomrule
    \end{tabular}
  \end{center}
\end{table}
%
%
\begin{table}
  \caption{Numerical results for DC1 problem with $cf=5.5$ using GMRES(30).}
  \label{T:large_cf_5.5}
  \begin{center}
    \begin{tabular}{llllllllll}
     \toprule
\addlinespace
      matrix & $h$ & \multicolumn{2}{c}{GMG-ND} & \multicolumn{2}{c}{GMG-NO}
& \multicolumn{2}{c}{EGMG-NO} & \multicolumn{2}{c}{AGMG} \\
\addlinespace
      \cmidrule(lr){3-4} \cmidrule(lr){5-6} \cmidrule(lr){7-8}
\cmidrule(lr){9-10}
\addlinespace
             &       & its                          & time                      
  & its & time & its & time & its & time  \\
\addlinespace
\midrule
\addlinespace

                & 1/800 & 51 & 37.0  & 33 & 23.5 & 33 & 23.9 & 27 & 6.3   \\
         2D & 1/1000 & 52 & 61.3  & 34 & 37.6 & 33 & 37.8 & 37 & 14.1  \\
                & 1/1200 & 50 & 84.2 & 33 & 54.3 & 32 & 55.3 & 37 & 18.2 \\
&  & &  &  &  \\
\addlinespace

                & 1/70 & 469 & 148.1  & 162 & 50.9  & 160 & 51.5 &  NC & NA \\
         3D & 1/100 & 219 & 222.9 & 148 & 146.2  & 109 & 117.0 &  NC & NA
\\
                & 1/120 & 221 & 393.7   & 42 & 83.4 & 28 & 95.4 & NC & NA
\\
\addlinespace
      \bottomrule
    \end{tabular}
  \end{center}
\end{table}
%
%
\begin{table}
  \caption{Numerical results for DC1 problem with $cf=6$ using GMRES(30).}
  \label{T:large_cf_6}
  \begin{center}
    \begin{tabular}{llllllllll}
     \toprule
\addlinespace
      matrix & $h$ & \multicolumn{2}{c}{GMG-ND} & \multicolumn{2}{c}{GMG-NO}
& \multicolumn{2}{c}{EGMG-NO} & \multicolumn{2}{c}{AGMG} \\
\addlinespace
      \cmidrule(lr){3-4} \cmidrule(lr){5-6} \cmidrule(lr){7-8}
\cmidrule(lr){9-10}
\addlinespace
             &       & its                          & time                      
  & its & time & its & time & its & time  \\
\addlinespace
\midrule
\addlinespace

                & 1/800 & 55 & 41.4  & 36 & 25.2 & 36 & 25.2 & 27 & 6.3   \\
         2D & 1/1000 & 56 & 65.8  & 36 & 40.0 & 36 & 40.0 & 37 & 14.1  \\
                & 1/1200 & 55 & 92.9 & 36 & 56.1 & 36 & 56.8 & 37 & 18.2 \\
&  & &  &  &  \\
\addlinespace

                & 1/70 & 568 & 183.2  & 316 & 98.9  & 354 & 115.9 &  NC & NA
\\
         3D & 1/100 & 208 & 213.5 & 103 & 100.4  & 94 & 101.1 &  NC & NA
\\
                & 1/120 & 167 & 308.1   & 43 & 83.7 & 41 & 95.5 &  NC & NA
\\
\addlinespace
      \bottomrule
    \end{tabular}
  \end{center}
\end{table}
%
%
\begin{table}
  \caption{Numerical results for DC1 for 2D and 3D problems
with $cf=7$ using GMRES(30)}
  \label{T:large_cf_7}
  \begin{center}
    \begin{tabular}{llllllllll}
     \toprule
\addlinespace
      matrix & $h$ & \multicolumn{2}{c}{GMG-ND} & \multicolumn{2}{c}{GMG-NO}
& \multicolumn{2}{c}{EGMG-NO} & \multicolumn{2}{c}{AGMG} \\
\addlinespace
      \cmidrule(lr){3-4} \cmidrule(lr){5-6} \cmidrule(lr){7-8}
\cmidrule(lr){9-10}
\addlinespace
             &       & its                          & time                      
  & its & time & its & time & its & time  \\
\addlinespace
\midrule
\addlinespace

                & 1/800 & 68 & 41.9  & 51 & 25.8 & 41 & 27.2 & 27 & 6.3   \\
         2D & 1/1000 & 68 & 73.5  & 41 & 36.9 & 41 & 40.9 & 37 & 14.1  \\
                & 1/1200 & 73 & 112.0 & 41 & 59.4 & 41 & 60.0 & 37 & 18.2 \\
&  & &  &  &  \\
\addlinespace

                & 1/70 & NC & NA  & 462 & 140.9  & 277 & 86.6 &  NC &
NA \\
         3D & 1/100 & 566 & 542.2 & 266 & 249.9  & 255 & 239.0 &  NC &
NA
\\
                & 1/120 & 173 & 314.5  & 72 & 127.5 & 57 & 114.5 &  NC & NA
\\
\addlinespace
      \bottomrule
    \end{tabular}
  \end{center}
\end{table}
%
%
\begin{table}
  \caption{Numerical results for DC1 for 2D and 3D problems
with $cf=8$ using GMRES(30)}
  \label{T:large_cf_8}
  \begin{center}
    \begin{tabular}{llllllllll}
     \toprule
\addlinespace
      matrix & $h$ & \multicolumn{2}{c}{GMG-ND} & \multicolumn{2}{c}{GMG-NO}
& \multicolumn{2}{c}{EGMG-NO} & \multicolumn{2}{c}{AGMG} \\
\addlinespace
      \cmidrule(lr){3-4} \cmidrule(lr){5-6} \cmidrule(lr){7-8}
\cmidrule(lr){9-10}
\addlinespace
             &       & its                          & time                      
  & its & time & its & time & its & time  \\
\addlinespace
\midrule
\addlinespace

                & 1/800 & 77 & 48.3  & 48 & 30.9 & 48 & 30.9 & 27 & 6.3   \\
         2D & 1/1000 & 76 & 74.6  & 46 & 46.8 & 46 & 47.2 & 37 & 14.1  \\
                & 1/1200 & 77 & 108.6 & 46 & 67.0 & 45 & 65.9 & 37 & 18.2 \\
&  & &  &  &  \\
\addlinespace

                & 1/70 & NC & NA  & 360 & 120.2  & 381 & 125.1 &  NC &
NA \\
         3D & 1/100 & NC & NA & 454 & 441.4  & 440 & 426.5 &  NC &
NA
\\
                & 1/120 & 486 & 831.2  & 186 & 319.5 & 187 & 326.2 &  NC & NA
\\
\addlinespace
      \bottomrule
    \end{tabular}
  \end{center}
\end{table}
%
%
\begin{table}
  \caption{Numerical results for DCC1 problem with $cf=2.5$ using
GMRES(30).}
  \label{T:large_con_cf_2.5}
  \begin{center}
    \begin{tabular}{llllllllll}
     \toprule
\addlinespace
      matrix & $h$ & \multicolumn{2}{c}{GMG-ND} & \multicolumn{2}{c}{GMG-NO}
& \multicolumn{2}{c}{EGMG-NO} & \multicolumn{2}{c}{AGMG} \\
\addlinespace
      \cmidrule(lr){3-4} \cmidrule(lr){5-6} \cmidrule(lr){7-8}
\cmidrule(lr){9-10}
\addlinespace
             &       & its                          & time                      
  & its & time & its & time & its & time  \\
\addlinespace
\midrule
\addlinespace

                & 1/800 & 54 & 45.3  & 49 & 36.6 & 18 & 14.2 & 27 & 6.3   \\
         2D & 1/1000 & 75 & 90.4  & 60 & 73.7 & 19 & 83.0 & 37 & 14.1  \\
                & 1/1200 & 51 & 101.4 & 40 & 61.1 & 21 & 165.6 & 37 & 18.2 \\
&  & &  &  &  \\
\addlinespace

                & 1/70 & 100 & 50.3  & 14 & 17.4  & 14 & 161.6 &  16 & 3.4 \\
         3D & 1/100 & 139 & 223.8 & 15 & 100.8  & ME & NA &  15 & 8.8
\\
                & 1/120 & 123 & 471.3   & 15 & 264.2 & ME & NA &  NC & NA
\\
\addlinespace
      \bottomrule
    \end{tabular}
  \end{center}
\end{table}
%
%
\begin{table}
  \caption{Numerical results for DCC1 problem with $cf=3$ using
GMRES(30).}
  \label{T:large_con_cf_3}
  \begin{center}
    \begin{tabular}{llllllllll}
     \toprule
\addlinespace
      matrix & $h$ & \multicolumn{2}{c}{GMG-ND} & \multicolumn{2}{c}{GMG-NO}
& \multicolumn{2}{c}{EGMG-NO} & \multicolumn{2}{c}{AGMG} \\
\addlinespace
      \cmidrule(lr){3-4} \cmidrule(lr){5-6} \cmidrule(lr){7-8}
\cmidrule(lr){9-10}
\addlinespace
             &       & its                          & time                      
  & its & time & its & time & its & time  \\
\addlinespace
\midrule
\addlinespace

                & 1/800 & 28 & 24.7  & 20 & 16.8 & 19 & 22.8 & 27 & 6.3   \\
         2D & 1/1000 & 29 & 50.4  & 19 & 25.6 & 19 & 25.5 & 37 & 14.1  \\
                & 1/1200 & 26 & 65.8 & 19 & 37.6 & 20 & 70.8 & 37 & 18.2 \\
&  & &  &  &  \\
\addlinespace

                & 1/70 & 113 & 48.7  & 15 & 10.6  & 15 & 46.9 &  16 & 3.4 \\
         3D & 1/100 & 153 & 207.3 & 17 & 53.2  & 17 & 681.8 &  15 & 8.8
\\
                & 1/120 & 135 & 337.8   & 17 & 119.7 & ME & ME &  NC & NA
\\
\addlinespace
      \bottomrule
    \end{tabular}
  \end{center}
\end{table}
%
%
\begin{table}
  \caption{Numerical results for DCC1 problem with $cf=3.5$ using
GMRES(30).}
  \label{T:large_con_cf_3.5}
  \begin{center}
    \begin{tabular}{llllllllll}
     \toprule
\addlinespace
      matrix & $h$ & \multicolumn{2}{c}{GMG-ND} & \multicolumn{2}{c}{GMG-NO}
& \multicolumn{2}{c}{EGMG-NO} & \multicolumn{2}{c}{AGMG} \\
\addlinespace
      \cmidrule(lr){3-4} \cmidrule(lr){5-6} \cmidrule(lr){7-8}
\cmidrule(lr){9-10}
\addlinespace
             &       & its                          & time                      
  & its & time & its & time & its & time  \\
\addlinespace
\midrule
\addlinespace

                & 1/800 & 47 & 36.5  & 32 & 25.1 & 27 & 24.7 & 27 & 6.3   \\
         2D & 1/1000 & 52 & 64.0  & 41 & 44.4 & 28 & 42.9 & 37 & 14.1  \\
                & 1/1200 & 44 & 80.8 & 30 & 57.8 & 27 & 65.5 & 37 & 18.2 \\
&  & &  &  &  \\
\addlinespace

                & 1/70 & 129 & 48.6  & 16 & 8.7  & 16 & 21.0 &  16 & 3.4 \\
         3D & 1/100 & 171 & 194.2 & 18 & 36.5  & 18 & 248.0 &  15 & 8.8
\\
                & 1/120 & 150 & 325.0   & 18 & 78.3 & 18 & 744.7 &  NC & NA
\\
\addlinespace
      \bottomrule
    \end{tabular}
  \end{center}
\end{table}
%
%
\begin{table}
  \caption{Numerical results for DCC1 for 2D and 3D problems
with $cf=4$ using GMRES(30)}
  \label{T:large_con_cf_4}
  \begin{center}
    \begin{tabular}{llllllllll}
     \toprule
\addlinespace
      matrix & $h$ & \multicolumn{2}{c}{GMG-ND} & \multicolumn{2}{c}{GMG-NO}
& \multicolumn{2}{c}{EGMG-NO} & \multicolumn{2}{c}{AGMG} \\
\addlinespace
      \cmidrule(lr){3-4} \cmidrule(lr){5-6} \cmidrule(lr){7-8}
\cmidrule(lr){9-10}
\addlinespace
             &       & its                          & time                      
  & its & time & its & time & its & time  \\
\addlinespace
\midrule
\addlinespace

                & 1/800 & 57 & 43.8  & 32 & 24.8 & 31 & 26.2 & 27 & 6.3   \\
         J2D & 1/1000 & 53 & 56.4  & 38 & 41.7 & 35 & 43.7 & 37 & 14.1  \\
                & 1/1200 & 52 & 81.3 & 39 & 61.9 & 38 & 66.1 & 37 & 18.2 \\
&  & &  &  &  \\
\addlinespace

                & 1/70 & 138 & 52.0  & 26 & 12.4  & 26 & 15.8 &  16 & 3.4 \\
         3D & 1/100 & 182 & 191.2 & 19 & 31.1  & 19 & 109.9 &  15 & 8.8
\\
                & 1/120 & 170 & 308.8  & 20 & 61.2 & 19 & 312.9 &  NC & NA
\\
\addlinespace
      \bottomrule
    \end{tabular}
  \end{center}
\end{table}

\clearpage
%
%
\begin{table}
  \caption{Numerical results for DCC1 for 2D and 3D problems
with $cf=4.5$ using GMRES(30)}
  \label{T:large_con_cf_4.5}
  \begin{center}
    \begin{tabular}{llllllllll}
     \toprule
\addlinespace
      matrix & $h$ & \multicolumn{2}{c}{GMG-ND} & \multicolumn{2}{c}{GMG-NO}
& \multicolumn{2}{c}{EGMG-NO} & \multicolumn{2}{c}{AGMG} \\
\addlinespace
      \cmidrule(lr){3-4} \cmidrule(lr){5-6} \cmidrule(lr){7-8}
\cmidrule(lr){9-10}
\addlinespace
             &       & its                          & time                      
  & its & time & its & time & its & time  \\
\addlinespace
\midrule
\addlinespace

                & 1/800 & 76 & 53.9  & 39 & 27.6 & 37 & 27.5 & 27 & 6.3   \\
         J2D & 1/1000 & 74 & 82.2  & 39 & 43.9 & 38 & 44.4 & 37 & 14.1  \\
                & 1/1200 & 64 & 112.6 & 39 & 64.1 & 37 & 65.6 & 37 & 18.2 \\
&  & &  &  &  \\
\addlinespace

                & 1/70 & 151 & 53.5  & 38 & 16.1  & 38 & 18.1 &  16 & 3.4 \\
         3D & 1/100 & 210 & 234.6 & 21 & 31.1  & 21 & 66.0 &  15 & 8.8
\\
                & 1/120 & 180 & 363.6  & 22 & 61.5 & 22 & 169.2 &  NC & NA
\\
\addlinespace
      \bottomrule
    \end{tabular}
  \end{center}
\end{table}
%
%
\begin{table}
  \caption{Numerical results for DCC1 for 2D and 3D problems
with $cf=5$ using GMRES(30)}
  \label{T:large_con_cf_5}
  \begin{center}
    \begin{tabular}{llllllllll}
     \toprule
\addlinespace
      matrix & $h$ & \multicolumn{2}{c}{GMG-ND} & \multicolumn{2}{c}{GMG-NO}
& \multicolumn{2}{c}{EGMG-NO} & \multicolumn{2}{c}{AGMG} \\
\addlinespace
      \cmidrule(lr){3-4} \cmidrule(lr){5-6} \cmidrule(lr){7-8}
\cmidrule(lr){9-10}
\addlinespace
             &       & its                          & time                      
  & its & time & its & time & its & time  \\
\addlinespace
\midrule
\addlinespace

                & 1/800 & 88 & 60.95  & 45 & 29.0 & 43 & 28.1 & 27 & 6.1   \\
         2D & 1/1000 & 80 & 77.2  & 48 & 48.6 & 49 & 50.8 & 37 & 13.2  \\
                & 1/1200 & 81 & 114.0 & 43 & 64.4 & 42 & 64.7 & 37 & 18.3 \\
&  & &  &  &  \\
\addlinespace

                & 1/70 & 169 & 59.4  & 60 & 22.8 & 60 & 24.0 &  16 & 3.4
\\
         3D & 1/100 & 224 & 219.8 & 22 & 29.7  & 22 & 45.3 &  15 & 8.8 \\
                & 1/120 & 195 & 340.35   & 23 & 56.8 & 23 & 111.0 &  NC & NA
\\
\addlinespace
      \bottomrule
    \end{tabular}
  \end{center}
\end{table}
%
%
\begin{table}
  \caption{Numerical results for DCC1 for 2D and 3D problems
with $cf=5.5$ using GMRES(30)}
  \label{T:large_con_cf_5.5}
  \begin{center}
    \begin{tabular}{llllllllll}
     \toprule
\addlinespace
      matrix & $h$ & \multicolumn{2}{c}{GMG-ND} & \multicolumn{2}{c}{GMG-NO}
& \multicolumn{2}{c}{EGMG-NO} & \multicolumn{2}{c}{AGMG} \\
\addlinespace
      \cmidrule(lr){3-4} \cmidrule(lr){5-6} \cmidrule(lr){7-8}
\cmidrule(lr){9-10}
\addlinespace
             &       & its                          & time                      
  & its & time & its & time & its & time  \\
\addlinespace
\midrule
\addlinespace

                & 1/800 & 107 & 73.0  & 50 & 33.7 & 48 & 32.7 & 27 & 6.1   \\
         2D & 1/1000 & 100 & 107.8  & 50 & 55.3 & 49 & 52.7 & 37 & 13.2  \\
                & 1/1200 & 90 & 148.7 & 48 & 74.2 & 47 & 72.8 & 37 & 18.3 \\
&  & &  &  &  \\
\addlinespace

                & 1/70 & 282 & 100.6  & 65 & 24.5 & 70 & 26.7 &  16 & 3.4
\\
         3D & 1/100 & 240 & 260.6 & 44 & 52.0  & 45 & 63.2 &  15 & 8.8 \\
                & 1/120 & 223 & 405.1   & 24 & 56.9 & 24 & 84.5 &  NC & NA
\\
\addlinespace
      \bottomrule
    \end{tabular}
  \end{center}
\end{table}
%
%
\begin{table}
  \caption{Numerical results for DCC1 for 2D and 3D problems
with $cf=6$ using GMRES(30)}
  \label{T:large_con_cf_6}
  \begin{center}
    \begin{tabular}{llllllllll}
     \toprule
\addlinespace
      matrix & $h$ & \multicolumn{2}{c}{GMG-ND} & \multicolumn{2}{c}{GMG-NO}
& \multicolumn{2}{c}{EGMG-NO} & \multicolumn{2}{c}{AGMG} \\
\addlinespace
      \cmidrule(lr){3-4} \cmidrule(lr){5-6} \cmidrule(lr){7-8}
\cmidrule(lr){9-10}
\addlinespace
             &       & its                          & time                      
  & its & time & its & time & its & time  \\
\addlinespace
\midrule
\addlinespace

                & 1/800 & 132 & 79.7  & 56 & 36.7 & 56 & 37.0 & 27 & 6.3   \\
         2D & 1/1000 & 117 & 112.4  & 55 & 56.1 & 54 & 55.2 & 37 & 14.1  \\
                & 1/1200 & 111 & 152.5 & 55 & 80.9 & 53 & 77.9 & 37 & 18.2 \\
&  & &  &  &  \\
\addlinespace

                & 1/70 & 261 & 88.9  & 90 & 32.2  & 88 & 31.5 &  16 & 3.4
\\
         3D & 1/100 & 256 & 250.4 & 41 & 47.1  & 41 & 51.3 &  15 & 8.8
\\
                & 1/120 & 225 & 384.6  & 26 & 61.7 & 27 & 78.9 &  NC & NA
\\
\addlinespace
      \bottomrule
    \end{tabular}
  \end{center}
\end{table}
%
%
\begin{table}
  \caption{Numerical results for DCC1 problem with $cf=7$ using
GMRES(30).}
  \label{T:large_con_cf_7}
  \begin{center}
    \begin{tabular}{llllllllll}
     \toprule
\addlinespace
      matrix & $h$ & \multicolumn{2}{c}{GMG-ND} & \multicolumn{2}{c}{GMG-NO}
& \multicolumn{2}{c}{EGMG-NO} & \multicolumn{2}{c}{AGMG} \\
\addlinespace
      \cmidrule(lr){3-4} \cmidrule(lr){5-6} \cmidrule(lr){7-8}
\cmidrule(lr){9-10}
\addlinespace
             &       & its                          & time                      
  & its & time & its & time & its & time  \\
\addlinespace
\midrule
\addlinespace

                & 1/800 & 184 & 112.3  & 67 & 42.7 & 65 & 42.0 & 27 & 6.3   \\
         2D     & 1/1000 & 173 & 162.6  & 69 & 68.1 & 67 & 67.5 & 37 & 14.1  \\
                & 1/1200 & 397 & 530.4 & 67 & 97.6 & 64 & 95.2 & 37 & 18.2 \\
&  & &  &  &  \\
\addlinespace

                & 1/70 & 351 & 117.5  & 109 & 32.6  & 109 & 36.3 &  16 & 3.4 \\
         3D     & 1/100 & 301 & 296.2 & 80 & 82.8  & 78 & 81.8 &  15 & 8.8
\\
                & 1/120 & 257 & 440.0   & 29 & 69.7 & 29 & 74.5 &  NC & NA
\\
\addlinespace
      \bottomrule
    \end{tabular}
  \end{center}
\end{table}
%
%
\begin{table}
  \caption{Numerical results for DCC1 problem with $cf=8$ using
GMRES(30).}
  \label{T:large_con_cf_8}
  \begin{center}
    \begin{tabular}{llllllllll}
     \toprule
\addlinespace
      matrix & $h$ & \multicolumn{2}{c}{GMG-ND} & \multicolumn{2}{c}{GMG-NO}
& \multicolumn{2}{c}{EGMG-NO} & \multicolumn{2}{c}{AGMG} \\
\addlinespace
      \cmidrule(lr){3-4} \cmidrule(lr){5-6} \cmidrule(lr){7-8}
\cmidrule(lr){9-10}
\addlinespace
             &       & its                          & time                      
  & its & time & its & time & its & time  \\
\addlinespace
\midrule
\addlinespace

                & 1/800 & 266 & 171.3  & 82 & 51.5 & 83 & 54.2 & 27 & 6.3   \\
         2D     & 1/1000 & 229 & 225.6  & 84 & 81.0 & 83 & 79.8 & 37 & 14.1  \\
                & 1/1200 & 218 & 297.7 & 81 & 112.4 & 78 & 107.8 & 37 & 18.2 \\
&  & &  &  &  \\
\addlinespace

                & 1/70 & 502 & 161.0  & 123 & 41.0  & 115 & 37.9 &  16 & 3.4 \\
         3D     & 1/100 & 449 & 445.9 & 105 & 102.4  & 104 & 101.8 &  15 & 8.8
\\
                & 1/120 & 294 & 517.2   & 57 & 109.8 & 57 & 112.2 &  NC & NA
\\
\addlinespace
      \bottomrule
    \end{tabular}
  \end{center}
\end{table}
%
%
\begin{table}
  \caption{Comparison of exact and inexact coarse operators. Here NO stands for natural ordering and 
ND stands for nested dissection ordering}
  \label{T:norm}
  \begin{center}
    \begin{tabular}{llcc}
      \toprule
\addlinespace
    Problem  &  1/$h$   &    $\frac{\|LU-\tilde{L}\tilde{U}\|}{\|LU\|}$ for NO &     $\frac{\|LU-\tilde{L}\tilde{U}\|}{\|LU\|}$ for ND  \\
\addlinespace
      \midrule
\addlinespace
	    & 1/30	& $7.5e^{-5}$        & $7.8e^{-5}$                                                    \\
    JUMP3D  & 1/40 	& $8.3e^{-5}$        & $7.2e^{-5}$                                                     \\
            & 1/50	& $7.0e^{-5}$        & $1.1e^{-4}$                                                    \\
\addlinespace
     \bottomrule
\addlinespace
    \end{tabular}
  \end{center}
\end{table}
%
%
\begin{table}
  \caption{Numerical results for Florida matrix market collection.}
  \label{table1}
  \begin{center}
    \begin{tabular}{llllllll}
      \toprule
\addlinespace
      matrix          & $N_c$  & \multicolumn{2}{c}{GMG-ND} & \multicolumn{2}{c}{GMG-NO} & \multicolumn{2}{c}{AGMG}       \\
\addlinespace
    \cmidrule(lr){3-4} \cmidrule(lr){5-6} \cmidrule(lr){7-8} 
               &                                     & its                          & time & its & time & its & time \\
\addlinespace
\midrule
\addlinespace

      gyto\_m         & 4096                                    & 111                          & 5.2 & 137 & 7.0 & $>600$ & NA \\
      bodyy4          & 4096                                    & 54                          & 1.8 & 19 & 0.4 & $>600$ & NA \\
      nd6k            & 4096                                    & NC         
                & NA & 89 & 21.3 & MEM & NA \\
      bodyy5          & 4096                                    & 115                          & 4.6 & 30 & 0.7 & - & - \\
      wathen100       & 4096                                    & 12                          & 1.4 & 9 & 0.7 & 11 & 4.1 \\
      wathen120       & 4096                                    & 12                          & 1.3 & 9 & 0.9 & 11 & 35.0 \\
      torsion1        & 4096                                    & 13                          & 0.99 & 9 & 0.6 & 6 & 0.1 \\
      obstclae        & 4096                                    & 13                          & 0.9 & 9 & 0.6 & 6 & 0.1 \\
      jnlbrng1        & 4096                                    & 22                          & 1.5 & 11 & 0.7 & 11 & 0.1 \\
      minsurfo        & 4096                                    & 16                          & 1.0 & 11 & 0.6 & 8 & 0.1 \\
      gridgena        & 4096                                    & 310                          & 81.4 & 189 & 34.0 & $>600$ & NA \\
      crankseg\_1     & 4096                                    & 60                          & 22.3 & 76 & 26.2 & 334 & 69.1 \\
      qa8fk           & 4096                                    & 11                          & 4.9 & 12 & 3.6 & 15 & 1.3 \\
      cfd1            & 4096                                    & 145                          & 45.1 & 127 & 32.0 & MEM & NA \\
      finan512        & 4096                                    & 7                          & 1.8 & 7 & 1.2 & 4 & 106.0 \\
      shallow\_water1 & 4096                                    & 6                          & 1.3 & 6 & 0.8 & 4 & 0.1 \\
      2cubes\_sphere  & 4096                                    & 6                          & 3.9 & 6 & 2.5 & 7 & 8.6 \\
      Thermal\_TC     & 4096                                    & 6                          & 2.03 & 7 & 1.4 & MEM & NA \\
      Thermal\_TK     & 4096                                    & 22                          & 3.6 & 23 & 3.5 & 80 & 5.0 \\
      G2\_circuit     & 4096                                    & 32                          & 7.7 & 24 & 4.2 & 91 & 5.9 \\
      bcsstk18        & 4096                                    & 232                          & 11.6 & 111 & 3.7 & $>600$ & NA \\
      cbuckle         & 4096                                    & 41                          & 2.3 & 62 & 3.0 & 493 & 18.7 \\
      Pres\_Poisson   & 4096                                    & 21                          & 1.7 & 24 & 1.7 & 24 & 16.8 \\
\addlinespace
      \bottomrule
    \end{tabular}
  \end{center}
\end{table}
\begin{table}
  \caption{Numerical results for Florida matrix market collection with $cf=2.5$}
  \label{T:table_florida2}
  \begin{center}
    \begin{tabular}{lllllllllll}
      \toprule
\addlinespace
      matrix            & \multicolumn{2}{c}{GMG-ND} &
\multicolumn{2}{c}{EGMG-ND} & \multicolumn{2}{c}{GMG-NO} & 
\multicolumn{2}{c}{EGMG-NO} & \multicolumn{2}{c}{AGMG}       \\
\addlinespace
    \cmidrule(lr){2-3} \cmidrule(lr){4-5} \cmidrule(lr){6-7} \cmidrule(lr){8-9}
\cmidrule(lr){10-11}
                                               & its                      
   & time & its & time & its & time & its & time & its & time \\
\addlinespace
\midrule
\addlinespace

      gyto\_m                                             & 271           
              & 4.1 & 319& 4.6 & 465 & 6.9 & 562 & 7.3 & NC & NA \\
      bodyy4                                             & 56            
             & 0.9 & 56& 0.9 & 19 & 0.3 & 19& 0.4& NC & NA \\
      nd6k                                               & NC         
                & NA &NC &NA & NC & NA & NC & NA & ME & NA \\
      bodyy5                                              & 135           
              & 1.8 &137 &1.8 & 31 & 0.6& 31 & 0.6 & 1 & 245.4 \\
      wathen100                                           & 12            
             & 1.0 & 12 & 1.1 & 9 & 0.6 &9 &0.7 & 11 & 4.3 \\
      wathen120                                           & 12            
             & 1.2 & 12&  1.3& 9 & 0.8 &9 &0.8 & 11 & 7.0 \\
      torsion1                                            & 14            
             & 0.9 &14 &1.0 & 10 & 0.5 &10 &0.6 & 6 & 0.1 \\
      obstclae                                           & 14            
             & 0.9 & 14& 1.0 & 10 & 0.5 & 10& 0.6& 6 & 0.1 \\
      jnlbrng1                                          & 25           
             & 1.3 &24 &1.3 & 12 & 0.6 &  12& 0.6& 11 & 0.2 \\
      minsurfo                                            & 18
             & 1.0 &18 & 1.2& 11 & 0.5 &12 &0.6 & 7 & 0.2 \\
      gridgena                                            & NC           
              & NA & NC & NA & 307 & 10.3 & 308 & 9.2 & NC & NA \\
      crankseg\_1                                         & 74            
             & 21.4 & 109 & 31.6 & 99 & 25.5 &87 & 29.0& 468 & 96.0 \\
      qa8fk                                              & 12            
             & 4.9 & 11 &5.8 & 11 & 3.6 &10 & 4.4& 15 & 1.3 \\
      cfd1                                                & 439           
              & 37.0 & 442  &38.3 & 259 & 23.0 & 255& 30.5 &  ME & NA \\
      finan512                                            & 6             
            & 1.8 & 6 &1.9 & 7 & 1.2 &7 &1.3 & 4 & 106.0 \\
      shallow\_water1                                     & 6             
            & 1.4 &6 &2.3 & 6 & 0.9 &6 &1.7 &  4 & 0.1 \\
      2cubes\_sphere                                      & 6             
            & 4.4 & 7 &8.6 & 5 & 3.0 & 5& 15.2&  7 & 8.6 \\
      Thermal\_TC                                        & 6             
            & 2.2 & 6& 2.5& 6 & 1.5 &6 &1.9 &  ME & NA \\
      Thermal\_TK                                         & NC            
             & NA & NC & NA & NC & NA & NC & NA & NC & NA \\
      G2\_circuit                                         & 28            
             & 7.3 & 21 & 7.7& 21 & 4.7 &15 &5.0 & 74 & 4.8 \\
      bcsstk18                                           & NC           
              & NA &NC  &NA & 166 & 1.9 &169 &2.0 &  F1 & NA \\
      cbuckle                                             & 54            
             & 1.5 & 54 & 1.5 & 247 & 4.0 & 198&3.7  & F1 & NA \\
      Pres\_Poisson                                       & 24            
             & 1.0 & 22&1.0  & 26 & 1.0 & 26 &1.0  & 24 & 21.2 \\
\addlinespace
      \bottomrule
    \end{tabular}
  \end{center}
\end{table}
\begin{figure}
  \caption{Convergence curve for Pres\_Poisson for exact and inexact
    coarse grid solves. CPU time indicated inside small box.}
  \label{F:curvePres}
  \begin{center}
    \includegraphics[scale=0.33]{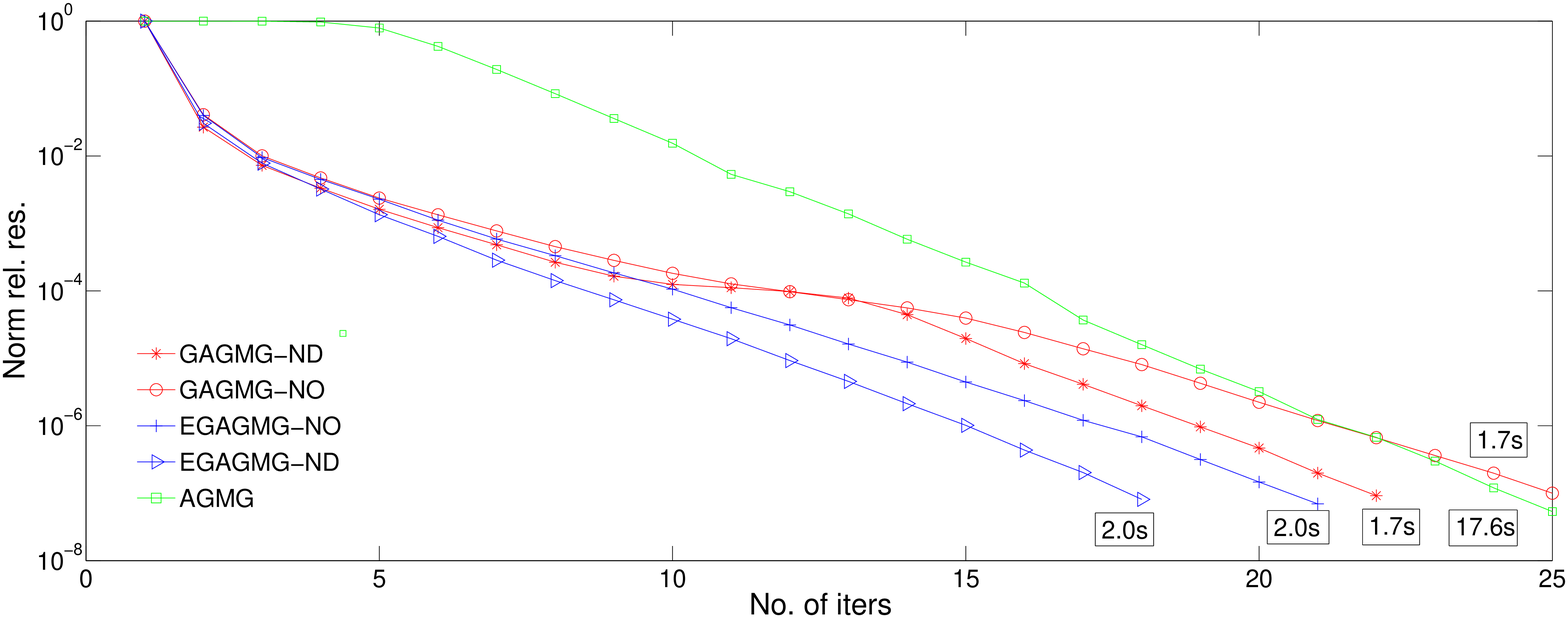}
  \end{center}
\end{figure}

\begin{figure}
  \caption{Convergence curve for DC3D $50 \times 50 \times 50$ for exact and inexact
    coarse grid solves. CPU time indicated inside small box.}
  \label{F:curveJUMP}
  \begin{center}
    \includegraphics[scale=0.33]{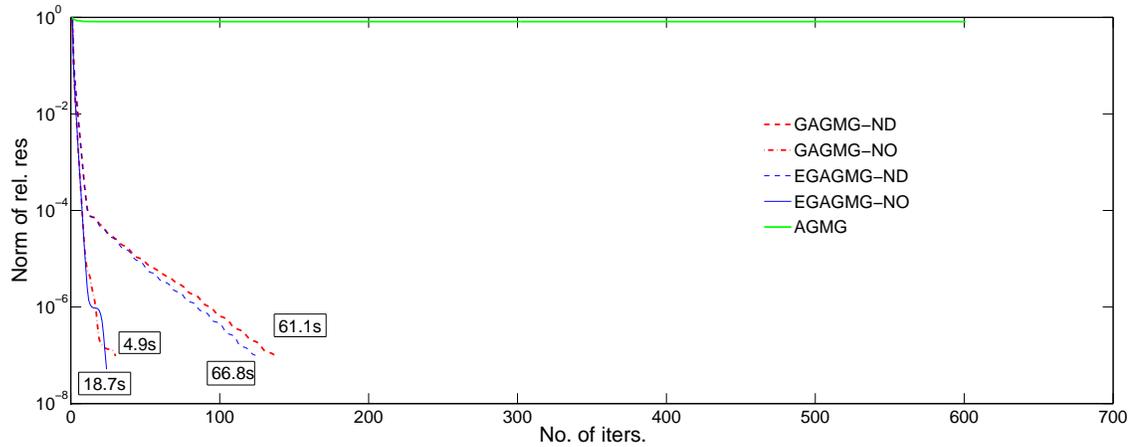}
  \end{center}
\end{figure}
\begin{figure}
  \caption{Comparison of no. of nonzeros for exact and inexact
    coarse grid solves for JUMP3D}
  \label{F:nonZeros}
  \begin{center}
    \includegraphics[scale=0.3]{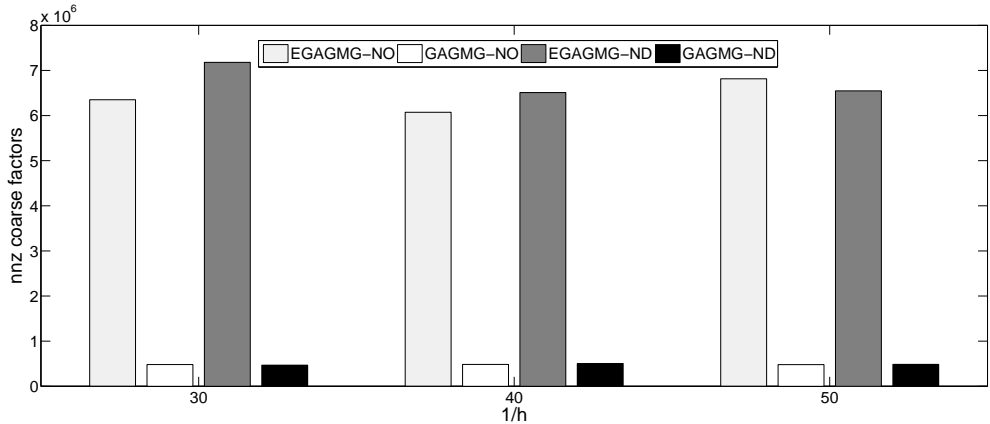}
  \end{center}
\end{figure}
\begin{figure}
  \caption{Jump in the diagonal entries of DC3D $30 \times 30 \times 30$
matrix when $\kappa(x)$ values are kept as $10^3$}
  \label{F:plotSky30Jump3}
  \begin{center}
    \includegraphics[scale=0.3]{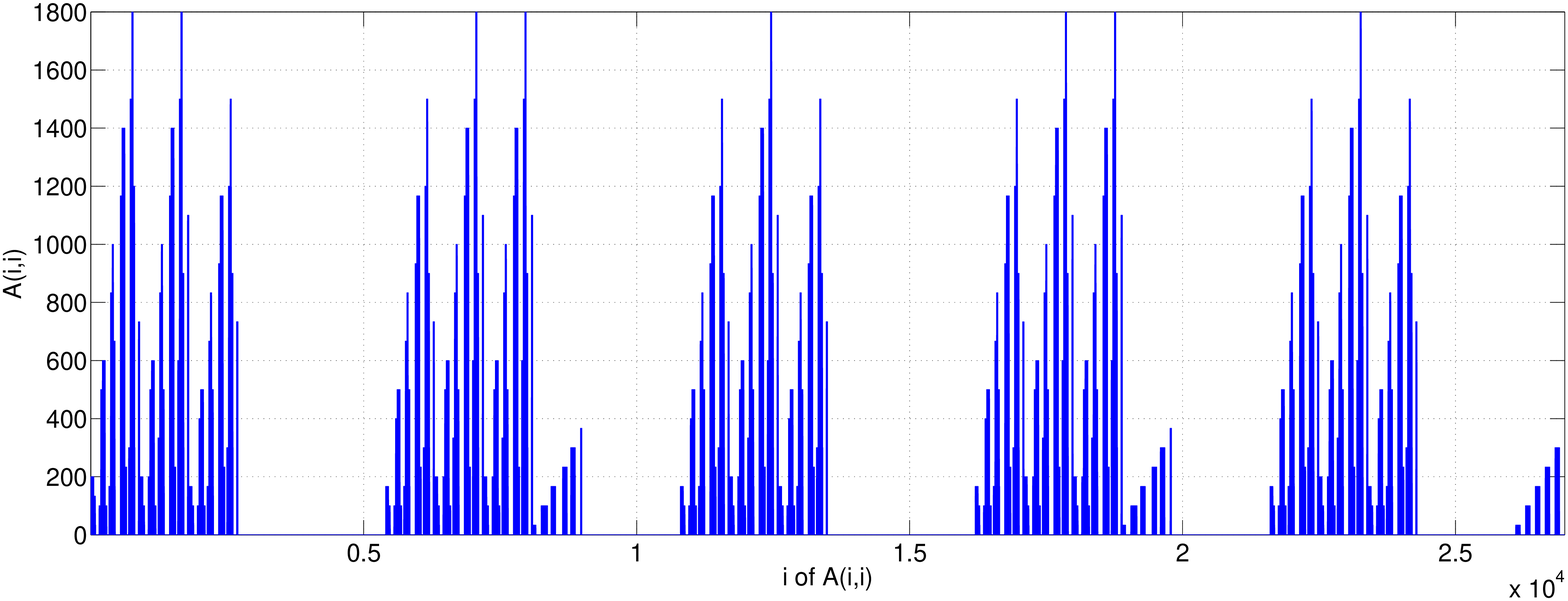}
  \end{center}
\end{figure}
%

\section{Conclusion}
We have proposed a two grid approach GAGMG with following ingredients
\vspace{-5mm}
\begin{packed_item}
\item Coarse grid based on graph matching
\item ILU(0) is the smoother with natural or nested dissection
  reordering
\item Coarse grid equation is solved inexactly
\end{packed_item}
\vspace{-5mm}
We compared out approach with the classical AGMG scheme. On
comparison, we found that the new strategy seems to be robust with a
very modest coarse grid size which is further solved cheaply by
performing an inexact solve. One of the aim of this work was to
provide a practical, easy to implement, yet robust two-grid methods.

We have tried only the sequential version of our method, in
future, we would like to implement the method in parallel with a
parallel inexact solve strategy. 
\section{Acknowledgement}
Many thanks to Universit\'e libre de Bruxelles for an ideal 
environment and the fond de la reserche scientifique (FNRS) 
Ref: 2011/V 6/5/004-IB/CS-15 that made this work possible.


\end{document}